\documentclass[12pt, reqno]{amsart}
\usepackage[margin=1in]{geometry}
\usepackage{amssymb,latexsym,amsmath,amscd,amsfonts}
\usepackage{latexsym}
\usepackage[mathscr]{eucal}
\usepackage{bm}
\usepackage{mathptmx}
\usepackage{amssymb}
\usepackage{amsthm}
\usepackage{dcolumn}
\usepackage[all]{xy}
\usepackage{enumitem}
\usepackage[utf8]{inputenc}

\def \qed {\hfill \vrule height6pt width 6pt depth 0pt}
\def\textmatrix#1&#2\\#3&#4\\{\bigl({#1 \atop #3}\ {#2 \atop #4}\bigr)}
\def\dispmatrix#1&#2\\#3&#4\\{\left({#1 \atop #3}\ {#2 \atop #4}\right)}
\newcommand{\beg}{\begin{equation}}
	\newcommand{\eeg}{\end{equation}}
\newcommand{\ben}{\begin{eqnarray*}}
	\newcommand{\een}{\end{eqnarray*}}

\newtheorem{thm}{Theorem}[section]
\newtheorem{cor}[thm]{Corollary}
\newtheorem{lem}[thm]{Lemma}

\numberwithin{equation}{section} \theoremstyle{definition}
\newtheorem{defn}[thm]{Definition}

\newtheorem{eg}[thm]{Example}

\newcommand{\HS}{\mathcal H}
\newcommand{\KS}{\mathcal K}
\newcommand{\C}{\mathbb{C}}

\newcommand{\T}{\mathbb{T}}

\newcommand{\Z}{\mathbb{Z}}
\newcommand{\N}{\mathbb{N}}

\newcommand{\ov}{\overline}
\newcommand{\la}{\left\langle}
\newcommand{\ra}{\right\rangle}

\begin{document}
	\title[Intertwining of $*$-regular $q$-isometric dilations]{Intertwining of $*$-regular $q$-isometric dilations}

	\author[Tomar] {NITIN TOMAR}
	
	\address[Nitin Tomar]{Theoretical Statistics and Mathematics Unit, Indian Statistical Institute, Kolkata, West Bengal-7000108, India.} \email{tomarnitin414@gmail.com}		
	
	
	
	
	\begin{abstract}
		A tuple $\underline{T}=(T_1,\dots,T_k)$ of contractions on a Hilbert space $\mathcal H$ is said to be \emph{$q$-commuting} with $\|q\|=1$ if there exists a family of scalars
		$
		q=\{q_{ij}\in\mathbb C : |q_{ij}|=1,\ q_{ij}=q_{ji}^{-1},\ 1\le i<j\le k\}
		$
		such that $T_iT_j=q_{ij}T_jT_i$ for $1\le i<j\le k$.
		In this article, we characterize $q$-commuting pairs of contractions with $\|q\|=1$ that admit a minimal $*$-regular $q$-isometric dilation. We present sufficient conditions for the commutant lifting theorem for such pairs. Moreover, sufficient conditions are obtained for a $q$-commuting triple of contractions with $\|q\|=1$ to admit a $q$-isometric dilation.
		\end{abstract}

	\maketitle

	\section{Introduction}
	
	\noindent Throughout the paper, all operators are bounded linear maps on complex Hilbert spaces. We write $\mathbb{C}$, $\mathbb{D}$ and $\mathbb{T}$ for the complex plane, unit disc $\{z: |z|<1\}$ and unit circle $\{z : |z|=1\}$ respectively, and $\mathcal{B}(\mathcal{H})$ for the algebra of operators on a Hilbert space $\mathcal{H}$. The symbols $\mathbb{N}$, $\mathbb{Z}$ and $\mathbb{Z}_+$ denote the sets of natural numbers, integers and non-negative integers respectively. 
	
	\smallskip 
	
	A \textit{contraction} is an operator of norm at most $1$. A classical theorem due to Sz.-Nagy \cite{NagyII} states that every contraction dilates to a unitary. Ando \cite{Ando} extended Nagy’s dilation theorem to pairs of commuting contractions. However, Parrott \cite{Par} showed that a unitary dilation need not exist for commuting triples of contractions. In this direction, Brehmer \cite{Brehmer} introduced a stronger notion of unitary dilation, called a regular unitary dilation, and characterized the commuting families of contractions that admit such dilations. Another fundamental result in dilation theory is the \emph{commutant lifting theorem}, which states that if \(T\) and \(T'\) are contractions on Hilbert spaces \(\mathcal{H}\) and \(\mathcal{H}'\), respectively, with minimal isometric dilations \(V\) on \(\mathcal{K}\) and \(V'\) on \(\mathcal{K}'\) respectively, and if an operator \(A : \mathcal{H} \to \mathcal{H}'\) intertwines \(T\) and \(T'\), i.e., \(AT = T'A\), then there exists an operator \(B : \mathcal{K} \to \mathcal{K}'\) with $\|B\|=\|A\|$ such that \(BV = V'B\) and \(AP_{\mathcal{H}} = P_{\mathcal{H}'} B\). In contrast, M\"{u}ller \cite{Muller} showed that the commutant lifting type theorem fails in general for two pairs of contractions even if the pairs admit regular unitary dilations. Furthermore, the authors of \cite{Gaspar, GasparII} present sufficient conditions ensuring the existence of such a lifting, where the notion of $*$-regular isometric dilation plays a crucial role. Recall from \cite{Gaspar, Timotin} that an isometric dilation $(V_1, \dotsc, V_k)$ of a commuting tuple $(T_1, \dotsc, T_k)$ of contractions on a Hilbert space $\HS$ is said to be \textit{$*$-regular} if for all $m_1, \dotsc, m_k \in \mathbb{Z}$, we have
\begin{align}\label{eqn_102}
	\left[T_1^{m_1^+} \dotsc T_k^{m_k^+}\right]\left[T_1^{*m_1^-}\dotsc T_k^{*m_k^-}\right]= P_\mathcal{H}\left[V_1^{*m_1^-}\dotsc V_k^{*m_k^-}\right]\left[V_1^{m_1^+} \dotsc V_k^{m_k^+}\right]\bigg|_\mathcal{H}.
\end{align}
Here, $m^+=\max \{m, 0\}$ and $m^-=-\min\{m, 0\}$ for $m \in \mathbb{Z}$. While classical dilation and lifting results mainly concern commuting families of contractions, subsequent work has extended these ideas to the non-commutative setting. In particular, many results were established for $q$-commuting contractions with $\|q\|=1$, e.g., see \cite{Bis:Pal:Sah, K.M., Maji, Pal, PalII, Sebestyen, N1} and the references therein.
	
	\begin{defn}
		A tuple $\underline{T}=(T_1, \dotsc, T_k)$ of operators acting on a Hilbert space $\mathcal{H}$ is said be \textit{$q$-commuting}  with $\|q\|=1$ if there exists a family of scalars $q=\{q_{ij} \in \T \ :  \ q_{ij}=q_{ji}^{-1}, \  1 \leq i < j \leq k\}$ such that $T_iT_j=q_{ij}T_jT_i$ for $1\leq i < j \leq k$. In addition, if $T_iT_j^*=\overline{q}_{ij} T_j^*T_i$, then $\underline{T}$ is said to be \textit{doubly $q$-commuting}. For a $q$-commuting pair $(T_1,T_2)$ with $\|q\|=1$ and $q=\{q_{12}\}$, we also refer to the pair as $q_{12}$-commuting.
	\end{defn}
	
Accordingly, we now recall the notion of isometric and unitary dilations in the $q$-commutative setting. Let $\underline{T}=(T_1, \dotsc, T_k)$ be a $q$-commuting tuple of contractions with $\|q\|=1$ acting on a Hilbert space $\HS$. A $q$-commuting tuple of isometries $\underline{V}=(V_1, \dotsc, V_k)$ with $\|q\|=1$, acting on a Hilbert space $\mathcal{K} \supseteq \HS$, is called a \textit{$q$-isometric dilation} of $\underline{T}$ if
		\[
		T_1^{m_1}\dotsc T_k^{m_k}=P_\HS V_1^{m_1}\dotsc V_k^{m_k}|_{\HS} \quad \text{for all  $m_1, \dotsc, m_k \in \Z_+$.}
		\]	
		The dilation is said to be \textit{minimal} if 
		$
		\mathcal{K}=\ov{\text{span}}\left\{V_1^{m_1}\dotsc V_k^{m_k}h : m_1, \dotsc, m_k \in \Z_+, h \in \HS \right\}.
		$
		If $\underline{V}$ consists of unitaries, then $\underline{V}$ is said to be a  \textit{$q$-unitary dilation}, of $\underline{T}$. In this case, the dilation is called minimal if
		$
		\mathcal{K}=\ov{\text{span}}\left\{V_1^{m_1}\dotsc V_k^{m_k}h : m_1, \dotsc, m_k \in \Z, \ h \in \HS \right\}.
		$

\medskip 

In this article, we primarily focus on $q$-commuting tuples of contractions with $\|q\|=1$ and commutant lifting theorem for pairs of contractions in the $q$-commutative setting. In this direction, Sebesty\'{e}n \cite{SebII} proved the classical commutant lifting theorem to the $q$-commutative setting when $q=-1$. Keshari and Mallick \cite{K.M.} obtained  generalization of the commutant lifting theorem and Ando’s dilation theorem in the $q$-commutative setting. In view of these results, one naturally  considers the following question.

\medskip 

\noindent \textbf{Question.}  Let $\underline{T}=(T_1, T_2)$ and $\underline{T}'=(T_1', T_2')$ be $q$-commuting pairs of contractions with $\|q\|=1$ on Hilbert spaces $\HS$ and $\HS'$, respectively. Suppose an operator $A: \HS \to \HS'$ \textit{intertwines} $\underline{T}$ and $\underline{T}'$, i.e., $AT_1=T_1'A$ and $AT_2=T_2'A$. Assume that $\underline{U}=(U_1, U_2)$ on $\KS$ and $\underline{U}'=(U_1', U_2')$ on $\KS'$ are minimal $q$-isometric dilations of $\underline{T}$ and $\underline{T}'$ respectively.  Can $A$ be lifted to an operator $B: \mathcal K \to \mathcal K'$, meaning $A^* = B^*|_{\mathcal H'}$ (or equivalently $AP_{\mathcal H} = P_{\mathcal H'} B$), that intertwines $\underline{U}$ and $\underline{U}' ?$

\medskip 

As mentioned earlier, M\"{u}ller \cite{Muller} showed in the commutative case that the lifting problem above may fail to have a solution, even when $\underline{U}$ and $\underline{U}'$ are regular unitary dilations. In this article, we provide sufficient conditions under which such a lifting exists in the $q$-commutative setting, extending some of the earlier works in the commutative case \cite{Gaspar, GasparII}. To do so, we recall from \cite{PalII, N1} the notions of regular and $*$-regular isometric dilations in the $q$-commutative setting. We say that a $q$-commuting tuple  $\underline{T}=(T_1, \dotsc, T_k)$ of contractions with $\|q\|=1$ acting on a Hilbert space $\mathcal{H}$ admits a \textit{regular $q$-unitary dilation} if there exist a Hilbert space $\mathcal{K} \supseteq \mathcal{H}$ and a $q$-commuting tuple $\underline{U}=(U_1, \dotsc, U_k)$ of unitaries on $\mathcal{K}$ such that
	\begin{align}\label{eqn_reg}
		\left[T_{{1}}^{*m_{{1}}^-}\dotsc T_{{k}}^{*m_{{k}}^-}\right]\left[T_{{1}}^{m_{{1}}^+}\dotsc T_{{k}}^{m_{{k}}^+}\right]=P_\mathcal{H}\left[U_{{1}}^{*m_{{1}}^-}\dotsc U_{{k}}^{*m_{{k}}^-}\right]\left[U_{{1}}^{m_{{1}}^+}\dotsc U_{{k}}^{m_{{k}}^+}\right]\bigg|_\mathcal{H},
	\end{align}
	for all $m_1, \dotsc, m_k \in\mathbb{Z}$. If the family $\underline{U}$ consists of isometries that satisfies \eqref{eqn_reg}, then $\underline{U}$ is said to be a \textit{regular $q$-isometric dilation}. Furthermore, a $q$-isometric dilation $(V_1, \dotsc, V_k)$ of a $q$-commuting tuple of contractions $(T_1, \dotsc, T_k)$ with $\|q\|=1$ acting on a Hilbert space $\HS$ is said to be \textit{$*$-regular} if for all $m_1, \dotsc, m_k \in \mathbb{Z}$, we have
	\[
	\left[T_1^{m_1^+} \dotsc T_k^{m_k^+}\right]\left[T_1^{*m_1^-}\dotsc T_k^{*m_k^-}\right]=\underset{1\leq i<j \leq k}{\prod}q_{ij}^{m_i^-m_j^+-m_i^+m_j^-} P_\mathcal{H}\left[V_1^{*m_1^-}\dotsc V_k^{*m_k^-}\right]\left[V_1^{m_1^+} \dotsc V_k^{m_k^+}\right]\bigg|_\mathcal{H}.
	\] 
A seminal result from \cite{Brehmer} states that a commuting tuple $\underline{T}=(T_1, \dotsc, T_k)$ of contractions has a regular unitary dilation if and only if it satisfies Brehmer's positivity condition, i.e.,
\begin{equation*}
	\Delta(u)=	\underset{\{\alpha_1, \dotsc, \alpha_m\} \subset u}{\sum}(-1)^m(T_{\alpha_1}\dotsc T_{\alpha_m})^*(T_{\alpha_1}\dotsc T_{\alpha_m}) \geq 0 \ \ \text{for every subset $u$ of $\{1, \dotsc, k\}$.}
\end{equation*}

The authors of \cite{PalII} extended this result in the $q$-commutative setting, proving the following.

\begin{thm}[\cite{PalII}, Theorems 1.7 \& 1.8]\label{thm_103} A $q$-commuting tuple of contractions with $\|q\|=1$ has a regular $q$-unitary dilation if and only if it satisfies Brehmer's positivity condition. In particular, every doubly $q$-commuting tuple of contractions admits a regular $q$-unitary dilation.
	\end{thm} 

It is evident from Theorem \ref{thm_103} that every doubly $q$-commuting tuple of contractions admits a minimal $q$-isometric dilation. A natural question arises if such minimal $q$-isometric dilations are also doubly $q$-commuting. Motivated by the works \cite{Gaspar, Timotin}, the author of \cite{N1} characterized such $q$-commuting minimal isometric dilations.

	\begin{thm}[\cite{N1}, Theorems 2.15 \& 2.16]\label{thm_104}
		Let $\underline{T}=(T_1, \dotsc, T_k)$ be a $q$-commuting tuple of contractions acting on a Hilbert space $\HS$. Then a minimal $q$-isometric dilation of $\underline{T}$ is doubly $q$-commuting if and only if it is $*$-regular. Furthermore, $\underline{T}$ is doubly $q$-commuting if and only if $\underline{T}$ has a minimal $q$-isometric dilation which is both regular and $*$-regular.
	\end{thm}

As shown in the above theorem, $*$-regular $q$-isometric dilations are essential in the study of $q$-commuting contractions with $\|q\|=1$. Motivated by this, we provide in Theorem \ref{thm_001} a characterization of $q$-commuting pairs of contractions with $\|q\|=1$ that admit a minimal $*$-regular $q$-isometric dilation, leading to a structure theorem for such pairs. In particular, we show that for any such pair $(T_1,T_2)$, the operator $T_2$ lifts to a contraction on the minimal isometric dilation space of $T_1$ that doubly $q$-commutes with the dilation of $T_1$. Building on this, we obtain in Theorem \ref{thm_intertw} a commutant lifting theorem for two $q$-commuting pairs of contractions with $\|q\|=1$, which is one of the main results of the article.

\medskip 

As mentioned earlier, Parrott \cite{Par} constructed an example of a commuting triple of contractions which does not admit a commuting isometric dilation, and therefore an analogous conclusion fails in the $q$-commutative setting. However, Theorem \ref{thm_104} guarantees that a doubly $q$-commuting triple $(T_1, T_2, T_3)$ of contractions admits a $q$-isometric dilation. This naturally leads to the question of whether the same conclusion holds for a strictly smaller subclass. As an application of Theorem \ref{thm_intertw}, we first present in Theorem \ref{thm_suff} a sufficient condition for a $q$-isometric dilation for a $q$-commuting triple of contractions with $\|q\|=1$. We then conclude this article by showing that the doubly $q$-commutative condition on the entire triple $(T_1, T_2, T_3)$ may be relaxed and replaced by requiring only that the pairs $(T_1, T_2)$ and $(T_1, T_3)$ be doubly $q$-commuting.

\section{Intertwining of $*$-regular $q$-isometric dilations}\label{sec_02}

\noindent For a given $q$-commuting pair of contractions $(T_1, T_2)$ with $\|q\|=1$ having a minimal $*$-regular $q$-isometric dilation, our first result establishes the lifting of $T_2$ to a contraction on the minimal isometric dilation space of $T_1$, where the lifted operator doubly $q$-commutes with the dilation of $T_1$. This also leads to a structure theorem for such pairs in terms of unitaries and completely non-unitary contractions. Recall that a contraction $T$ on a Hilbert space $\mathcal H$ is said to be completely non-unitary (c.n.u.) if there is no nonzero reducing subspace of $\mathcal H$ on which $T$ acts as a unitary.

\begin{thm}\label{thm_001}
Let $\underline{T}=(T_1, T_2)$ be a $q$-commuting pair of contractions with $\|q\|=1$ acting on a Hilbert space $\HS$. Then the following are equivalent: 
\begin{enumerate}
\item $\underline{T}$ has a minimal $*$-regular $q$-isometric dilation;
\item If $W_1$ acting on a Hilbert space $\KS_1$ is the minimal isometric dilation of $T_1$, then there exists a contraction $W_2$ on $\KS_1$ such that $(W_1, W_2)$ is doubly $q$-commuting and
\[
(T_1^*, T_2^*)=(W_1^*|_{\HS}, W_2^*|_{\HS}).
\]
In particular, $(T_1^*, T_2^*)$ is the restriction to a joint invariant subspace of an operator
\[
\bigoplus_{\eta \subseteq \{1,2 \}} \underline{W}^*_\eta \in \mathcal{B}\left( \bigoplus_{\eta \subseteq \{1,2 \}} \mathcal{K}_{1\eta}\right),
\] 
where $\underline{W}_\eta^*=(W_1^*|_{\KS_{1\eta}}, W_2^*|_{\KS_{1\eta}})$. Moreover, $W_i^*|_{\KS_{1\eta}}$ is a unitary if $i \in \eta$, and a c.n.u. contraction if $i \notin \eta$ for each $\eta$. Also, $(W_1^*|_{\KS_{1\eta}}, W_2^*|_{\KS_{1\eta}})$ is doubly $q$-commuting for each $\eta$.
\end{enumerate}
\end{thm}
	
	\begin{proof}
		By $q$-commutativity of $\underline{T}$, we have that $T_1T_2=q_{12}T_2T_1$ for some $q_{12} \in \T$. Suppose $\underline{T}$ has a minimal $*$-regular $q$-isometric dilation. By Theorem \ref{thm_104}, $\underline{T}$ admits a minimal $q$-isometric dilation $\underline{V}=(V_1, V_2)$ on a Hilbert space $\mathcal{K}_+$, which is doubly $q$-commuting.  By the minimality condition, 
		$
		\KS_+=\overline{\text{span}}\left\{V_1^m V_2^n x: x \in \HS, m, n \in \Z_+\right\}.
		$
	It follows from Lemma 2.11 in \cite{N1} that $(T_1^*, T_2^*)=(V_1^*|_{\HS}, V_2^*|_{\HS})$. Let $W_1$ acting on a Hilbert space $\KS_1$ be the minimal isometric dilation of $T_1$. Again by minimality, $\KS_1=\overline{\text{span}}\{V_1^mx: x \in \HS, m \in \Z_+\}$ and $W_1=V_1|_{\KS_1}$. Here, we have used the fact that the minimal isometric dilation of a contraction is unique (up to a unitary equivalence). Define $W_2=P_{\KS_1}V_2|_{\KS_1}$, which is a contraction on $\KS_1$. Since $(T_1^*, T_2^*)=(V_1^*|_{\HS}, V_2^*|_{\HS})$, we have 
	\begin{align}\label{eqn_WT}
	W_i^*x=P_{\KS_1}V_i^*x=P_{\KS_1}T_i^*x=T_i^*x \quad \text{for all $x \in \HS$ and so,} \quad W_i^*|_{\HS}=V_i^*|_{\HS}=T_i^*
	\end{align}
for $i=1, 2$. We now prove that $W_2^*W_1=\overline{q}_{21}W_1W_2^*$. For $m \in \Z_+$ and $x \in \HS$, we have that 
\begin{align*}
	W_2^*W_1V_1^mx=W_2^*V_1^{m+1}x=P_{\KS_1}V_2^*V_1^{m+1}x
	=\overline{q}_{21}^{m+1}P_{\KS_1}V_1^{m+1}V_2^*x
	=\overline{q}_{21}^{m+1}V_1^{m+1}T_2^*x,
\end{align*}
where the last equality holds since $T_2^*=V_2^*|_{\HS}$ and $V_1^{m+1}T_2^*x \in \KS_1$. Furthermore, 
\begin{align*}
	W_1W_2^*V_1^mx=W_1P_{\KS_1}V_2^*V_1^{m}x=\overline{q}_{21}^mW_1P_{\KS_1}V_1^{m}V_2^*x
	=\overline{q}_{21}^{m}W_1P_{\KS_1}V_1^{m}T_2^*x 
	=\overline{q}_{21}^{m}W_1V_1^{m}T_2^*x 
	=\overline{q}_{21}^{m}V_1^{m+1}T_2^*x,
\end{align*}
where the last equality holds as $V_1^{m}T_2^*x \in \KS_1$. Thus, $W_2^*W_1V_1^mx=\overline{q}_{21}W_1W_2^*V_1^mx$ and by continuity, $W_2^*W_1=\overline{q}_{21}W_1W_2^*$. Next, we show that $W_2^*W_1^*=q_{21}W_1^*W_2^*$. Let $x, h \in \HS$ and $n \geq 1$. Then
\begin{align*}
	W_2^*W_1^*(h+V_1^nx)
	&=T_2^*T_1^*h+W_2^*W_1^*W_1^nx \quad [\text{by \eqref{eqn_WT}}]\\
	&=q_{21}T_1^*T_2^*h+W_2^*W_1^{n-1}x\\
	&=q_{21}T_1^*T_2^*h+\overline{q}_{21}^{n-1}W_1^{n-1}W_2^*x \quad [\text{since $W_2^*W_1=\overline{q}_{21}W_1W_2^*$}]\\
&=q_{21}W_1^*W_2^*h+\overline{q}_{21}^{n-1}W_1^*(W_1^{n}W_2^*)x \quad [\text{by \eqref{eqn_WT}}]\\
&=q_{21}W_1^*W_2^*h+\overline{q}_{21}^{n-1}q_{21}^nW_1^*W_2^*W_1^{n}x \quad [\text{since $W_1W_2^*=q_{21}W_2^*W_1$}]\\
&=q_{21}W_1^*W_2^*(h+V_1^nx),
\end{align*}
and thus by continuity arguments, it follows that $W_2^*W_1^*=q_{21}W_1^*W_2^*$. Consequently, $(W_1, W_2)$ is a doubly $q$-commuting pair of contractions on $\KS_1$. It follows from Wold type decomposition for doubly $q$-commuting contractions (see \cite{Maji, Pal}) that there exist $4$ joint reducing subspaces $\{\KS_{1\eta}: \eta \subseteq \{1,2 \}\}$ for $W_1^*, W_2^*$ such that 
\[
\KS=\bigoplus_{\eta \subseteq \{1,2 \}} \mathcal{K}_{1\eta} \quad \text{and} \quad \underline{W}^*=\bigoplus_{\eta \subseteq \{1,2 \}} \underline{W}^*_\eta,
\]
where $\underline{W}_\eta^*=(W_1^*|_{\KS_{1\eta}}, W_2^*|_{\KS_{1\eta}})$. Moreover, $W_i^*|_{\KS_{1\eta}}$ is a unitary if $i \in \eta$, and a c.n.u. contraction if $i \notin \eta$ for each $\eta$. Also, $(W_1^*|_{\KS_{1\eta}}, W_2^*|_{\KS_{1\eta}})$ is doubly $q$-commuting for every $\eta \subseteq \{1,2 \}$.

\medskip 

Conversely, let $W_1$ acting on a Hilbert space $\KS_1$ be the minimal isometric dilation of $T_1$, and let $(W_1, W_2)$ be a doubly $q$-commuting pair of contractions on $\KS_1$ such that $(T_1^*, T_2^*)=(W_1^*|_{\HS}, W_2^*|_{\HS})$. By minimality, $\KS_1=\overline{\text{span}}\{W_1^mx: x \in \HS, m \in \Z_+\}$. For $m \in \Z_+$ and $x \in \HS$, it is evident that 
\[
T_1^{*m}T_2^{*}=W_1^{*m}W_2^{*}|_{\HS} \quad \text{and so,} \quad P_{\HS}W_2(W_1^mx)=T_2T_1^mx=T_2P_{\HS}W_1^mx.
\]
Thus, $P_{\HS}W_2=T_2P_{\HS}$ and $P_{\HS}W_2W_1=T_2T_1P_{\HS}$.  By doubly $q$-commutativity of $(W_1, W_2)$, we have
\begin{align}\label{eqn_W12}
P_{\HS}W_1W_1^*W_2W_2^*x=P_{\HS}W_2W_1W_1^*W_2^*x=T_2T_1P_{\HS}W_1^*W_2^*x
=T_2T_1T_1^*T_2^*x.
\end{align}
The Brehmer's positivity condition for the $q$-commuting pair $(T_1^*, T_2^*)$ reduces to 
\begin{align*}
\Delta_{\{1, 2\}}(T_1^*, T_2^*)&=I-T_1T_1^*-T_2T_2^*+T_2T_1T_1^*T_2^*
=P_{\HS}(I-W_1W_1^*)(I-W_2W_2^*)|_{\HS},
\end{align*}
where the last equality follows from \eqref{eqn_W12} and thus, $(T_1^*, T_2^*)$ satisfies the Brehmer's positivity condition. It follows from Theorem \ref{thm_103} that $(T_1^*, T_2^*)$ admits  a minimal regular $q$-unitary dilation. We have by Corollary 2.10 in \cite{N1} that $(T_1, T_2)$ admits a minimal $*$-regular $q$-isometric dilation.  
\end{proof}
 
The following example shows that if a $*$-regular $q$-isometric dilation is replaced by a regular one, then the conclusion of the above theorem may not hold.

\begin{eg}\label{eg_1}
	Let $\{e_1, e_2, e_3\}$ be the standard orthonormal vectors in $\C^3$. Consider the operator pair $\underline{T}=(T_1, T_2)$ on $\HS=\C^3$ defined as 
	\[
T_1(\alpha e_1+ \beta e_2+ \gamma e_3)=\gamma e_1 \quad \text{and} \quad 	T_2(\alpha e_1+ \beta e_2+ \gamma e_3)=\beta e_1
\]
	for $\alpha, \beta, \gamma \in \C$. Some routine computation shows that $T_1T_2=T_2T_1=0$ and so, $\underline{T}$ is a $q$-commuting pair of contractions with $\|q\|=1$ for any choice of $q=\{q_{12}\}$ in $\T$. Moreover, we have
	\[
T_1^*(\alpha e_1+ \beta e_2+ \gamma e_3)=\alpha e_3 \quad \text{and} \quad 	T_2^*(\alpha e_1+ \beta e_2+ \gamma e_3)=\alpha e_2
\]
for $\alpha, \beta, \gamma \in \C$. It is not difficult to see that $\underline{T}$ satisfies the Brehmer's positivity condition and by Theorem \ref{thm_103}, $\underline{T}$ admits a minimal regular $q$-isometric dilation. Suppose $W_1$ acting on a Hilbert space $\KS_1$ is the minimal isometric dilation of $T_1$. Let there exist a contraction $W_2$ on $\KS_1$ such that $(W_1, W_2)$ is doubly $q$-commuting and
$(T_1^*, T_2^*)=(W_1^*|_{\HS}, W_2^*|_{\HS})$. By Theorem \ref{thm_001}, $\underline{T}$ has a minimal $*$-regular $q$-isometric dilation and so, $(T_1^*, T_2^*)$ has a regular $q$-isometric dilation (see Lemma 2.9 and Corollary 2.10 in \cite{N1}). By Theorem \ref{thm_103}, $(T_1^*, T_2^*)$ satisfies Brehmer's positivity condition. Thus, 
$
\la \left(I-T_1T_1^*-T_2T_2^*+(T_1^*T_2^*)^*(T_1^*T_2^*)\right)e_1, e_1 \ra= \la -e_1, e_1\ra \geq 0,
$
which is a contradiction. \qed
\end{eg}

We now present one of the main results of this article providing sufficient conditions for the commutant lifting theorem for $q$-commuting pairs of contractions with $\|q\|=1$, extending earlier work in the commutative case \cite{Gaspar, GasparII}.

\begin{thm}\label{thm_intertw}
	Let $\underline{T}=(T_1, T_2)$ and $\underline{T}'=(T_1', T_2')$ be $q$-commuting pairs of contractions with $\|q\|=1$ acting on Hilbert spaces $\HS$ and $\HS'$ respectively, which admit minimal $*$-regular $q$-isometric dilations. Suppose $\underline{U}=(U_1, U_2)$ acting on a Hilbert space $\mathcal{K}$ and $\underline{U}'=(U_1', U_2')$ acting on a Hilbert space $\mathcal{K}'$ are minimal $*$-regular $q$-isometric dilations of $\underline{T}$ and $\underline{T}'$ respectively. If a contraction $A: \HS \to \HS'$ intertwines $\underline{T}$ and $\underline{T}'$ such that $AT_1^*=T_1'^*A$, then there exists an intertwining contraction $B: \mathcal{K} \to \mathcal{K}'$ for $\underline{U}$ and $\underline{U}'$ such that	$A^*=B^*|_{\HS'}, \|B\|=\|A\|$ and $BU_1^*=U_1'^*B$.
\end{thm}

\begin{proof}
The minimality of $\underline{U}$ and $\underline{U}'$ gives that
\[
\KS=\overline{\text{span}}\left\{U_1^mU_2^nx: x \in \HS, m, n \in \Z_+ \right\} \quad \text{and} \quad \KS'=\overline{\text{span}}\left\{U_1'^mU_2'^nx': x' \in \HS', m, n \in \Z_+ \right\}.
\] 
Furthermore, it follows from Theorem \ref{thm_104} that $\underline{U}$ and $\underline{U}'$ are doubly $q$-commuting. For $\KS_1=\overline{\text{span}}\left\{U_1^mx: x \in \HS, m \in \Z_+\right\}$ and $\KS_1'=\overline{\text{span}}\left\{U_1'^m x': x' \in \HS', m \in \Z_+\right\}$, we define
\[
(S_1, S_2)=(U_1|_{\KS_1}, P_{\KS_1}U_2|_{\KS_1}) \quad \text{and} \quad (S_1', S_2')=(U_1'|_{\KS'_1}, P_{\KS'_1}U_2'|_{\KS'_1}).
\]
Evidently, $S_1$ and $S_1'$ are the minimal isometric dilations of $T_1$ and $T_1'$ respectively. It follows from the proof of Theorem \ref{thm_001} that $(S_1, S_2)$ and $(S_1', S_2')$ are doubly $q$-commuting pair of contractions and the following holds:
\begin{align}\label{eqn_203}
(S_1^*|_{\HS}, S_2^*|_{\HS})=(T_1^*, T_2^*)=(U_1^*|_{\HS}, U_2^*|_{\HS}) \ \ \text{and} \ \  (S_1'^*|_{\HS'}, S_2'^*|_{\HS'})=(T_1'^*, T_2'^*)=(U_1'^*|_{\HS'}, U_2'^*|_{\HS'}),
\end{align}
where both the last equalities follow from Lemma 2.11 in \cite{N1}. From here onwards, the proof is divided into several steps for better understanding.

\medskip 

\noindent \textbf{\textit{Step 1.}} In this step, we construct a contraction $A_0$ between the minimal isometric dilation spaces of $T_1$ and $T_1'$, and show that $A_0$ intertwines the corresponding isometric dilations. The minimal isometric dilations spaces can be written (see \cite{NagyBook}, Chapter II ) as follows:  
\[
\KS_1=\HS \oplus \bigoplus_{p \in \Z_+}S_1^p \ \overline{(S_1-T_1)\HS} \quad \text{and} \quad \KS_1'=\HS' \oplus \bigoplus_{p \in \Z_+}S_1'^p \ \overline{(S_1'-T_1')\HS'}.
\]
Let $D_{T_1}=I-T_1^*T_1$ and $D_{T_1'}=I-T_1'^*T_1'$. By the intertwining property of $A$, we have that $D_{T_1'}^2A=AD_{T_1}^2$ and so, $D_{T_1'}A=AD_{T_1}$. For $x \in \HS$ and $p \geq 0$, it now follows that
\begin{align*}
	\|S_1^p(S_1-T_1)x\|^2=\|(S_1-T_1)x\|^2
	&=\la (S_1-T_1)x, (S_1-T_1)x \ra \\
	&=\|x\|^2-\la x, S_1^*T_1x\ra-\la S_1^*T_1x, x\ra+\|T_1x\|^2\quad [\text{since $S_1$ is an isometry}]\\
	&= \|x\|^2-\la x, T_1^*T_1x\ra-\la T_1^*T_1x, x\ra+\|T_1x\|^2 \quad [\text{as $S_1^*=T_1^*|_{\HS}$}]\\
	&=\|x\|^2-\|T_1x\|^2\\
	&=\|D_{T_1}x\|^2.
\end{align*}
Similarly, one can show that $\|S_1'^p(S_1'-T_1')Ax\|^2=\|(S_1'-T_1')Ax\|^2=\|D_{T_1'}Ax\|^2=\|AD_{T_1x}\|^2$, where the last equality holds since $D_{T_1'}A=AD_{T_1}$. Consequently, we have
\begin{align}\label{eqn_204}
\|S_1'^p(S_1'-T_1')Ax\|^2=\|AD_{T_1x}\|^2 \leq \|D_{T_1}x\|^2=\|S_1^p(S_1-T_1)x\|^2.
\end{align}
Consider the operator $A_0: \KS_1 \to \KS_1'$ given by
\[
A_0\left(x+ \sum_{p \in \Z_+}S_1^p(S_1-T_1)x_p\right)=Ax+ \sum_{p \in \Z_+}S_1'^p(S_1'-T_1')Ax_p.
\] 
It follows from \eqref{eqn_204} that  $A_0$ is a well-defined contraction. Evidently, $A_0|_{\HS}=A$. Moreover, $\|A_0\|=\|A\|$. To see this, let $\{x, x_p : p \in \Z_+\} \subset \HS$. Then we have by \eqref{eqn_204} that
\begin{align*}
\left\|A_0\left(x+ \sum_{p \in \Z_+}S_1^p(S_1-T_1)x_p\right)\right\|^2
&=\left\|Ax\right\|^2+\sum_{p \in \Z_+}\left\|S_1'^p(S_1'-T_1')Ax_p\right\|^2\\
&=\left\|Ax\right\|^2+\sum_{p \in \Z_+}\left\|AD_{T_1}x_p\right\|^2\\
& \leq \|A\| \left(\|x\|^2+\sum_{p \in \Z_+}\left\|D_{T_1}x_p\right\|^2\right)\\
&=\|A\| \left\|x+ \sum_{p \in \Z_+}S_1^p(S_1-T_1)x_p\right\|^2.
\end{align*}
We show that $A_0$ intertwines $\underline{S}$ and $\underline{S}'$, and $A_0S_1^*=S_1'^*A_0$. For $p \in \Z_+$ and $x, x_p \in \HS$, note that
\begin{align*}
A_0S_1\left(x+ S_1^p(S_1-T_1)x_p\right)
&=A_0\left(T_1x+ (S_1-T_1)x+ S_1^{p+1}(S_1-T_1)x_p\right)\\
&=AT_1x+(S_1'-T_1')Ax+S_1'^{p+1}(S_1'-T_1')Ax_p\\
&=T_1'Ax+(S_1'-T_1')Ax+S_1'^{p+1}(S_1'-T_1')Ax_p \quad [\text{as $AT_1=T_1'A$}]\\
&=S_1'Ax+S_1'^{p+1}(S_1'-T_1')Ax_p\\
&=S_1'A_0\left(x+ S_1^p(S_1-T_1)x_p\right).
\end{align*} 
Let $m \in \N$ and $x, x_0, x_m \in \HS$. Since $A_0|_\HS=A, S_1^*|_\HS=T_1^*$ and $S_1$ is an isometry, we have
\begin{align*}
A_0S_1^*\left(x+ (S_1-T_1)x_0+S_1^m(S_1-T_1)x_m\right)
&=A_0T_1^*x+A_0x_0-A_0T_1^*T_1x_0+A_0S_1^{m-1}(S_1-T_1)x_m\\
&=AT_1^*x+Ax_0-AT_1^*T_1x_0+A_0S_1^{m-1}(S_1-T_1)x_m\\
&=T_1'^*Ax+Ax_0-T_1'^*T_1'Ax_0+S_1'^{(m-1)}(S_1'-T_1')Ax_m\\
& \qquad [\text{since $ AT_1^*=T_1'^*A, \ AT_1=T_1'A, \ A_0S_1=S_1'A_0$}]\\
&=S_1'^*Ax+S_1'^*(S_1'-T_1')Ax_0+S_1'^*S_1'^m(S_1'-T_1')Ax_m\\
&=S_1'^*A_0\left(x+ (S_1-T_1)x_0+S_1^m(S_1-T_1)x_m\right).
\end{align*}
Next, we show that $A_0^*|_{\HS'}=A^*$. For $x' \in \HS'$ and $x, x_p \in \HS$ with $p \in \Z_+$, we have
\begin{align*}
	\la A_0^*x', x+ S_1^p(S_1-T_1)x_p \ra 
	&= \la x', A_0 x\ra + \la  x', A_0S_1^p(S_1-T_1)x_p\ra \\
	&=\la x', Ax\ra + \la  x', S_1'^p(S_1'-T_1')Ax_p\ra  \quad [\text{as $A_0|_{\HS}=A$}]\\
	&=\la x', Ax\ra + \la  A^*(S_1'^*-T_1'^*)S_1'^{*p}x',x_p\ra\\
	&=\la x', Ax\ra + \la  (S_1^*-T_1^*)S_1^{*p}A^*x',x_p\ra
	\quad [\text{since $S_1'^*|_{\HS'}=T_1'^*$ and $AT_1=T_1'A$}]\\
	&=\la A^*x', x+ S_1^p(S_1-T_1)x_p \ra.
\end{align*}
We prove that $A_0S_2=S_2'A_0$. Let $x', x_p' \in \HS'$ and $p \in \Z_+$. Since $(A_0^*|_{\HS'}, S_2'^*|_{\HS'})=(A^*, T_2'^*)$, it follows that 
$
A_0^*S_2'^*x'=A^*T_2'^*x'=T_2^*A^*x'=S_2^*A_0^*x'.
$
Since $(S_1', S_2')$ is doubly $q$-commuting, we have
\begin{align*}
	A_0^*S_2'^*\left(S_1'^pS_1'x_p'\right)
	&=\overline{q}_{21}^{p+1}A_0^*S_1'^{(p+1)}S_2'^*x_p'\\
	&=\overline{q}_{21}^{p+1}A_0^*S_1'^{(p+1)}T_2'^*x_p'\\
&=\overline{q}_{21}^{p+1}S_1^{(p+1)}S_2^*A^*x_p' \quad [\text{as $S_1'^*A_0=A_0S_1^*, A_0^*|_{\HS'}=A^*$ and $T_2'A=AT_2$}]\\
&=S_2^*S_1^{(p+1)}A_0^*x_p' \quad [\text{since $A_0|_{\HS'}=A$ and $(S_1', S_2')$ is a doubly $q$-commuting pair}]\\
&=S_2^*A_0^*S_1'^{(p+1)}x_p',\quad  \text{and} \\
	A_0^*S_2'^*\left(S_1'^pT_1'x_p'\right)
	&=\overline{q}_{21}^{p}A_0^*S_1'^{p}S_2'^*T_1'x_p'\\
	&=\overline{q}_{21}^{p}S_1^{p}A_0^*T_2'^*T_1'x_p' \quad [\text{since $S_2'^*|_{\HS'}=T_2'^*$ and $A_0S_1^*=S_1'^*A_0$}]\\
	&=\overline{q}_{21}^{p}S_1^{p}T_2^*A^*T_1'x_p' \quad [\text{as $A_0^*|_{\HS'}=A^*$ and $AT_2=T_2'A$}]\\
	&=\overline{q}_{21}^{p}S_1^{p}S_2^*A^*T_1'x_p' \\
	&=S_2^*S_1^{p}A_0^*T_1'x_p' \quad [\text{as $A_0^*|_{\HS'}=A^*$ and $(S_1', S_2')$ is a doubly $q$-commuting pair}]\\
		&=S_2^*A_0^*S_1'^{p}T_1'x_p'.
\end{align*}
Consequently, $S_2^*A_0^*=A_0^*S_2'^*$ and so, $A_0S_2=S_2'A_0$. Putting everything together, we have that
\begin{align}\label{eqn_step2}
A_0S_1=S_1'A_0, \ \ A_0S_1^*=S_1'^*A_0, \ \  A_0S_2=S_2'A_0, \ \ (A_0|_{\HS}, A_0^*|_{\HS'})=(A, A^*) \ \ \text{and} \ \ \|A_0\|=\|A\|.
\end{align}

\medskip 

\noindent \textbf{\textit{Step 2.}} In this step, we prove the existence of an operator $B : \KS \to \KS'$ with $\|B\|=\|A\|$ such that $B$ intertwines $U_2$ and $U_2'$, and $A^*=B^*|_{\HS'}$. To begin with, it follows from \eqref{eqn_step2} that $AP_{\HS}=P_{\HS'}A_0$. For $x \in \HS$ and $m \in \Z_+$,  we have by \eqref{eqn_203} that 
\begin{align*}
S_2^*U_1^mx=P_{\KS_1}U_2^*U_1^mx
=\ov{q}_{21}^mP_{\KS_1}U_1^mU_2^*x
=\ov{q}_{21}^mP_{\KS_1}U_1^mT_2^*x 
=\ov{q}_{21}^mU_1^mT_2^*x
=\ov{q}_{21}^mU_1^mU_2^*x
=U_2^*U_1^mx
\end{align*}
and so, $S_2^*=U_2^*|_{\KS_1}$. It follows from the definitions of $\KS_1$ and $\KS$ that  \[
\KS=\ov{\text{span}}\left\{U_1^mU_2^nx: x \in \HS, m, n \in \Z_+ \right\}=\ov{\text{span}}\left\{U_2^nk_1: k_1 \in \KS_1, n \in \Z_+\right\}.
\] 
Therefore, $U_2$ on $\KS$ is the minimal isometric dilation of $S_2$, and similarly one can show that $U_2'$ on $\KS'$ is the minimal isometric dilation of $S_2'$. For $x, x_m \in \HS$ with $m \in \N$, it follows from \eqref{eqn_203} that
\begin{align*}
	S_1^*\left(x+U_1^mx_m\right)
	=P_{\KS_1}U_1^*\left(x+U_1^mx_m\right)
	=P_{\KS_1}T_1^*x+P_{\KS_1}U_1^{m-1}x_m
	=T_1^*x+U_1^{m-1}x_m
	=U_1^*(x+U_1^mx_m)
	\end{align*} 
and so, $S_1^*=U_1^*|_{\KS_1}$. Similarly, it follows that $S_1'^*=U_1'^*|_{\KS_1'}$. Since $U_2'$ on $\KS'$ is the minimal isometric dilation of $S_2'$, we can re-write the structure of $\KS'$ as follows: 
\[
\KS'=\KS_1' \oplus \bigoplus_{n \in \Z_+}U_2'^n\ov{(U_2'-S_2')\KS'_1}.
\]
Putting everything together, $S_2$ on $\KS_1$ and $S_2'$ on $\KS_1'$ are contractions, and $U_2$ on $\KS$ and $U_2'$ on $\KS'$ are their minimal isometric dilations respectively. Also, the operator $A_0: \KS_1 \to \KS_1'$ intertwines $S_2$ and $S_2'$. It follows from Theorem 2.3 of Chapter II in \cite{NagyBook} that there exists an operator $B : \KS \to \KS'$ such that $\|B\|=\|A_0\|, BU_2=U_2'B$ and $A_0^*=B^*|_{\KS_1'}$. By \eqref{eqn_step2}, $\|B\|=\|A\|$ and $A^*=B^*|_{\HS'}$. 

\medskip 

\noindent \textbf{\textit{Step 3.}} We now explicitly recall the construction of the operator $B$ from Section 2 of Chapter II of \cite{NagyBook}, with further related discussion in Section 3 of \cite{GasparII}. The operator $B:\KS\to\KS'$ is defined as
\[
Bk=A_0P_{\KS_1}k+\overset{\infty}{\underset{n=0}{\sum}}U_2'^nB_nk \quad \text{($k \in \KS$)}.
\]
Here, the contractions $B_n: \KS \to \overline{(U_2'-S_2')\KS_1'}$ are inductively defined as $B_n=D_nC_n$ with 
\begin{align*}
& C_n: \KS \to \KS, \quad C_0=(I_\KS-A_0^*A_0P_{\KS_1})^{1\slash 2} \quad \text{and} \quad  C_n=(C_0^2-B_0^*B_0-\dotsc-B_{n-1}^*B_{n-1})^{1\slash 2},
\quad \text{and} \\
& D_n: \KS \to  \overline{(U_2'-S_2')\KS_1'},  \quad D_n(C_nU_2k)=B_{n-1}k \ \ (k \in \KS) \quad  \text{and} \quad D_n=0 \ \text{on} \ \KS \ominus \ov{C_nU_2\KS},
\end{align*} 
where $B_{-1}=(U_2'-S_2')A_0P_{\KS_1}$. It is not difficult to see that $B(\KS \ominus \HS) \subseteq \KS' \ominus \HS', P_{\HS'}B=AP_{\HS}$ and $\|B\|=\|A\|$. Since $B_nU_2=B_{n-1}$ for $n \in \N$, it follows that $BU_2=U_2'B$. 

\medskip 

\noindent \textbf{\textit{Step 4.}} Finally, we establish that $B$ is the desired operator and it only remains to show that $BU_1=U_1'B$ and $BU_1^*=U_1'^*B$. To do so, we consider the intertwining relations of $B$ with $U_n$ and $U_n'$, and with $U_n^*$ and $U_n'^*$ for each $n$. Since $S_1=U_1|_{\KS_1}, S_1^*=U_1^*|_{\KS_1}, A_0S_1=S_1'A_0$ and $A_0S_1^*=S_1'^*A_0$, it follows that $C_0U_1=U_1C_0$ and $C_0U_1^*=U_1^*C_0$. Using $S_1'=U_1'|_{\KS_1'}$ and $S_1'^*=U_1'^*|_{\KS_1'}$, we have 
\begin{align*}
U_1'(U_2'-S_2')k_1'
&=U_1'U_2'k_1'-U_1'S_2'k_1'=q_{12}U_2'U_1'k_1'-S_1'S_2'k_1'=q_{12}(U_2'-S_2')S_1'k_1'\quad \text{and} \\
U_1'^*(U_2'-S_2')k_1'
&=U_1'^*U_2'k_1'-U_1'^*S_2'k_1'=\ov{q}_{12}U_2'U_1'^*k_1'-S_1'^*S_2'k_1'=\ov{q}_{12}(U_2'-S_2')S_1'^*k_1'
\end{align*}
for all $k_1' \in \KS'_1$. Consequently, $U_1'\ov{(U_2'-S_2')\KS'_1} \subseteq \ov{(U_2'-S_2')\KS'_1}, \ U_1'^*\ov{(U_2'-S_2')\KS'_1} \subseteq \ov{(U_2'-S_2')\KS'_1}$ and so, the subspace $\ov{(U_2'-S_2')\KS'_1}$ is a reducing subspace for $U_1'$. Then
\begin{align*}
	B_{-1}U_1
	=(U_2'-S_2')A_0P_{\KS_1}U_1
	&=(U_2'-S_2')A_0S_1P_{\KS_1} \quad [\text{since $S_1^*=U_1^*|_{\KS_1}$}]\\
	&=(U_2'-S_2')S_1'A_0P_{\KS_1} \quad [\text{as $A_0S_1=S_1'A_0$}]\\
	&=U_2'U_1'A_0P_{\KS_1}-S_2'S_1'A_0P_{\KS_1} \quad [\text{because $S_1'=U_1'|_{\KS_1'}$}]\\
	&=q_{21}(U_1'U_2'-S_1'S_2')A_0P_{\KS_1}\\
	&=q_{21}U_1'(U_2'-S_2')A_0P_{\KS_1} \quad [\text{as $S_1'=U_1'|_{\KS_1'}$}]\\
	&=q_{21}U_1'B_{-1}
\end{align*}
and
\begin{align*}
	B_{-1}U_1^*
	=(U_2'-S_2')A_0P_{\KS_1}U_1^*
	&=(U_2'-S_2')A_0S_1^*P_{\KS_1} \quad [\text{since $S_1=U_1|_{\KS_1}$}]\\
	&=(U_2'-S_2')S_1'^*A_0P_{\KS_1} \quad [\text{as $A_0S_1^*=S_1'^*A_0$}]\\
	&=U_2'U_1'^*A_0P_{\KS_1}-S_2'S_1'^*A_0P_{\KS_1} \quad [\text{because $S_1'^*=U_1'^*|_{\KS_1'}$}]\\
	&=\ov{q}_{21}(U_1'^*U_2'-S_1'^*S_2')A_0P_{\KS_1}\\
	&=\ov{q}_{21}U_1'^*(U_2'-S_2')A_0P_{\KS_1} \quad [\text{as $S_1'^*=U_1'^*|_{\KS_1'}$}]\\
	&=\ov{q}_{21}U_1'^*B_{-1}.
\end{align*}

Next, using $U_1U_2=q_{12}U_2U_1, U_1U_2^*=\overline{q}_{12}U_2^*U_1, C_0U_1=U_1C_0$ and $C_0U_1^*=U_1^*C_0$, we have
\begin{align*}
U_1\ov{C_0U_2\KS}
&\subseteq \ov{U_1C_0U_2\KS}=\ov{C_0U_1U_2\KS} \subseteq \ov{C_0U_2U_1\KS} \subseteq \ov{C_0U_2\KS} \quad \text{and} \\
U_1^*\ov{C_0U_2\KS}
&\subseteq \ov{U_1^*C_0U_2\KS}=\ov{C_0U_1^*U_2\KS} \subseteq \ov{C_0U_2U_1^*\KS} \subseteq \ov{C_0U_2\KS}.
\end{align*}
Therefore, $\ov{C_0U_2\KS}$ is a reducing subspace for $U_1$ and thus, $U_1$ doubly commutes with $P_0$, where $P_0$ is the orthogonal projection of $\KS$ onto $\ov{C_0U_2\KS}$. Then
\[
D_0U_1(I-P_0)=D_0(I-P_0)U_1=0=U_1'D_0(I-P_0) \  \ \text{and similarly, } \ \ D_0U_1^*(I-P_0)=0=U_1'^*D_0(I-P_0) .
\]
So, $D_0U_1=U_1'D_0$ and $D_0U_1^*=U_1'^*D_0$ on $\KS \ominus \ov{C_0U_2\KS}$. For $k \in \KS$, it follows that 
\begin{align*}
	D_0U_1(C_0U_2k)
	&=D_0C_0U_1U_2k \quad [\text{as $C_0U_1=U_1C_0$}]\\
	&=q_{12}D_0C_0U_2U_1k\\
	&=q_{12}B_{-1}U_1k\\
	&=U_1'B_{-1}k \quad [\text{since $B_{-1}U_1=q_{21}U_1'B_{-1}$}]\\
	&=U_1'D_0C_0U_2k
\end{align*}
and 
\begin{align*}
	D_0U_1^*(C_0U_2k)
	&=D_0C_0U_1^*U_2k \quad [\text{as $C_0U_1^*=U_1^*C_0$}]\\
	&=\ov{q}_{12}D_0C_0U_2U_1^*k\\
	&=\ov{q}_{12}B_{-1}U_1^*k\\
	&=U_1'^*B_{-1}k \quad [\text{since $B_{-1}U_1^*=\ov{q}_{21}U_1'^*B_{-1}$}]\\
	&=U_1'^*D_0C_0U_2k.
\end{align*}
Thus, $D_0U_1=U_1'D_0$ and $D_0U_1^*=U_1'^*D_0$ on $\KS$. It now follows that
\[
B_0U_1=D_0C_0U_1=D_0U_1C_0=U_1'D_0C_0=U_1'B_0 \quad \text{and} \quad  B_0U_1^*=D_0C_0U_1^*=D_0U_1^*C_0=U_1'^*B_0.
\]
Note that $B_{-1}U_1=q_{12}^{-1}U_1'B_{-1}, B_{-1}U_1^*=\ov{q}_{12}^{-1}U_1'^*B_{-1}, B_0U_1=U_1'B_0$ and $B_0U_1^*=U_1'^*B_0$. Following the similar techniques as above, one can easily show that 
\begin{align}\label{eqn_205}
B_nU_1=(q_{12})^nU_1'B_n \quad \text{and} \quad B_nU_1^*=(\ov{q}_{12})^nU_1'^*B_n
\end{align}
for $n \geq -1$. Since $P_{\KS_1}U_1=S_1P_{\KS_1}$ and $P_{\KS_1}U_1^*=S_1^*P_{\KS_1}$, we have that
\begin{align*}
	BU_1
	&=A_0P_{\KS_1}U_1+\overset{\infty}{\underset{n=0}{\sum}}U_2'^nB_nU_1\\
	&=A_0S_1P_{\KS_1}+\overset{\infty}{\underset{n=0}{\sum}}(q_{12})^nU_2'^nU_1'B_n \quad [\text{by \eqref{eqn_205}}]\\
	&=S_1'A_0P_{\KS_1}+\overset{\infty}{\underset{n=0}{\sum}}(q_{12})^n (q_{21})^nU_1'U_2'^nB_n \quad [\text{since $(U_1', U_2')$ is doubly $q$-commuting}]\\
	&=U_1'A_0P_{\KS_1}+\overset{\infty}{\underset{n=0}{\sum}}U_1'U_2'^nB_n  \quad [\text{as $S_1'=U_1'|_{\KS_1'}$}]\\
	&=U_1'B
\end{align*}
and similarly, it follows that $BU_1^*=U_1'^*B$. The proof is now complete.
\end{proof}

As an application of Theorem \ref{thm_intertw}, we now present sufficient conditions that ensure a $q$-isometric dilation for $q$-commuting triples of contractions with $\|q\|=1$. Before this, we put forth a basic lemma whose proof follows directly from Theorem \ref{thm_103}. 

\begin{lem}\label{lem_205}
	Let $(V_1, V_2, V_3)$ be a $q$-commuting triple of contractions with $\|q\|=1$ on a Hilbert space $\mathcal{H}$ such that $(V_1, V_2)$ is doubly $q_{12}$-commuting,
	 $(V_1, V_3)$ is doubly $q_{13}$-commuting and $V_2$ is an isometry. Then $(V_1, V_2, V_3)$ admits a minimal regular $q$-isometric dilation.	
\end{lem}

\begin{proof}
	It is not difficult to see that $(V_1, V_2, V_3)$ satisfies the Brehmer's positivity condition. The desired conclusion now follows directly from Theorem \ref{thm_103}. 
\end{proof}

The following theorem shows that a $q$-commuting triple of contractions with $\|q\|=1$ admits a $q$-isometric dilation, provided that two of the three pairs admit $*$-regular $q$-isometric dilation. 

\begin{thm}\label{thm_suff}
	Let $(T_1, T_2, T_3)$ be a $q$-commuting triple of contractions with $\|q\|=1$ on a Hilbert space $\HS$. If the pairs $(T_1, T_2)$ and $(T_1, T_3)$ admit minimal $*$-regular $q_{12}$-isometric and $q_{13}$-isometric dilations respectively, then $(T_1, T_2, T_3)$ has a $q$-isometric dilation.
\end{thm}

\begin{proof}
	Let $S_1$ on a Hilbert space $\KS_1$ be the minimal isometric dilation of $T_1$. By minimality, $\KS_1=\ov{\text{span}}\left\{S_1^mx: x \in \HS, m \in \Z_+\right\}$. By Theorem \ref{thm_001}, there exist contractions $S_2, S_3$ on $\KS_1$ such that $(S_1, S_2)$ is doubly $q_{12}$-commuting and $(S_1, S_3)$ is doubly $q_{13}$-commuting, and that $(T_1^*, T_2^*, T_3^*)=(S_1^*|_{\HS}, S_2^*|_{\HS}, S_3^*|_{\HS})$. Then
	$
	S_2^*S_3^*S_1^mx=\ov{q}_{21}^m\ov{q}_{31}^mS_1^mS_2^*S_3^*x=\ov{q}_{21}^m\ov{q}_{31}^mS_1^mT_2^*T_3^*x=q_{23}\ov{q}_{21}^m\ov{q}_{31}^mS_1^mT_3^*T_2^*x=q_{23}S_3^*S_2^*S_1^mx
	$
for all $m \in \Z_+$ and $x \in \HS$. By continuity arguments, $S_2S_3=q_{23}S_3S_2$ on $\KS_1$. Since $(S_1, S_2)$ is a doubly $q_{12}$-commuting pair of contractions on $\KS_1$, we have by Theorem \ref{thm_104} that it admits a minimal $*$-regular $q_{12}$-isometric dilation $(U_1, U_2)$ on a Hilbert space $\KS_2$, which is doubly $q_{12}$-commuting. By minimality, $(U_1^*|_{\HS}, U_2^*|_{\HS})=(T_1^*, T_2^*)$. Let us define 
\[
A:=S_3, \quad (T_1', T_2'):=(q_{31}S_1, q_{32}S_2) \ \text{on} \ \KS_1 \quad \text{and} \quad (U_1', U_2'):=(q_{31}U_1, q_{32}U_2) \ \text{on} \ \KS_2.
\]
Evidently, $(T_1', T_2')$ is a $q_{12}$-commuting pair of contractions and $(U_1', U_2')$ is a doubly $q_{12}$-commuting pair of isometries, which is a minimal $*$-regular $q_{12}$-isometric dilation of $(T_1', T_2')$. Note that $A$ is a contraction that satisfies $AS_1=T_1'A, AS_2=T_2'A$ and $AS_1^*=T_1'^*A$. By Theorem \ref{thm_intertw}, there exists a contraction $U_3: \KS_2 \to \KS_2$ such that 
$U_3U_1=U_1'U_3, U_3U_2=U_2'U_3, U_3U_1^*=U_1'^*U_3$ and $U_3^*|_{\KS_1}=S_3^*$. Consequently, $U_3U_1=q_{31}U_1U_3, \ U_3U_2=q_{32}U_2U_3, U_3U_1^*=\ov{q}_{31}U_1^*U_3$ and $U_3^*|_{\HS}=S_3^*|_{\HS}=T_3^*$. Hence, the $q$-commuting triple $\underline{U}$ satisfies the hypothesis of Lemma \ref{lem_205} and so, $(U_1, U_2, U_3)$ admits a minimal $q$-isometric dilation. Then $(U_1^*, U_2^*, U_3^*)$ has a $q$-unitary dilation and so, $(T_1^*, T_2^*, T_3^*)=(U_1^*|_{\HS}, U_2^*|_{\HS}, U_3^*|_{\HS})$ admits a $q$-unitary dilation. Thus, $(T_1, T_2, T_3)$ has a $q$-isometric dilation. 
\end{proof}

Now, we present a dual version of the above theorem in terms of regular $q$-isometric dilation.

\begin{cor}
	Let $(T_1, T_2, T_3)$ be a $q$-commuting triple of contractions with $\|q\|=1$ on a Hilbert space $\HS$. If the pairs $(T_1, T_2)$ and $(T_1, T_3)$ admit minimal regular $q_{12}$-isometric and $q_{13}$-isometric dilations respectively, then $(T_1, T_2, T_3)$ has a $q$-isometric dilation.
\end{cor}
	
\begin{proof}
	It is easy to see that the pairs $(T_1^*, T_2^*)$ and $(T_1^*, T_3^*)$ admit minimal $*$-regular $q_{12}$-isometric and $q_{13}$-isometric dilations respectively (see Lemma 2.9 and Corollary 2.10 in \cite{N1}). By Theorem \ref{thm_suff}, $(T_1^*, T_2^*, T_3^*)$ has a $q$-isometric dilation. Since a $q$-commuting family of isometries admits a dilation to a $q$-commuting family of unitaries (see \cite{PalII}), the desired conclusion follows.
\end{proof}	

By Theorem \ref{thm_104}, a doubly $q$-commuting triple $(T_1,T_2,T_3)$ admits a $q$-isometric dilation. The following corollary shows that this assumption can be weakened by requiring only the pairs $(T_1,T_2)$ and $(T_1,T_3)$ to be doubly $q_{12}$-commuting and doubly $q_{13}$-commuting respectively. Before stating this, we present an example of such a triple, which is not doubly $q$-commuting. 

\begin{eg}
	Let $q=\{q_{ij} \in \T : 1 \leq i< j \leq 1 \}$ and let $(A, B)$ be a $q_{23}$-commuting pair of contractions on a Hilbert space $\HS$, which is not doubly $q_{23}$-commuting. It is not difficult to see that the pair in Example \ref{eg_1} satisfies this property. Consider the operators 
	\[
	T_1=\begin{pmatrix}
		A & 0 \\
		0 & 0
	\end{pmatrix}, \quad	T_2=\begin{pmatrix}
	0 & 0 \\
	0 & A
	\end{pmatrix} \quad \text{and} \quad	T_3=\begin{pmatrix}
	0 & 0 \\
	0 & B
	\end{pmatrix} \quad \text{on $\HS\oplus \HS$}
	\]
	Evidently, $(T_1, T_2, T_3)$ is a $q$-commuting triple of contractions such that $(T_1, T_2)$ and $(T_1, T_3)$ are doubly $q_{12}$-commuting isometric and doubly $q_{13}$-commuting respectively. Since $(A, B)$ is not doubly $q_{23}$-commuting, it follows that $(T_2, T_3)$ is not doubly $q_{23}$-commuting. \qed
	\end{eg}
 
\begin{cor}
	Let $(T_1, T_2, T_3)$ be a $q$-commuting triple of contractions on a Hilbert space $\HS$. If the pairs $(T_1, T_2)$ and $(T_1, T_3)$ are doubly $q_{12}$-commuting and doubly $q_{13}$-commuting respectively, then $(T_1, T_2, T_3)$ has a $q$-isometric dilation.
\end{cor} 

\begin{proof}
	By Theorem \ref{thm_104}, the pairs $(T_1, T_2)$ and $(T_1, T_3)$ have minimal $*$-regular $q_{12}$-isometric and $q_{13}$-isometric dilations respectively. The desired conclusion now follows from Theorem \ref{thm_suff}.
\end{proof}

	\subsection*{Acknowledgements}
During the course of this work, the author was supported via the J. C. Bose fellowship of Prof. Debashish Goswami, Indian Statistical Institute, Kolkata. The author greatly appreciates the warm and generous hospitality provided by the Indian Statistical Institute, Kolkata, India, and expresses sincere thanks to Prof. Debashish Goswami for his kind support.

\end{document}